\documentclass[12pt]{article}
\usepackage{lineno,hyperref}
\modulolinenumbers[5]

\usepackage{amsmath}
\usepackage{amsthm}
\usepackage{amssymb}
\usepackage{amsfonts}
\usepackage{fullpage}

\newtheorem{theorem}{Theorem}[section]
\newtheorem{definition}[theorem]{Definition}
\newtheorem{corollary}[theorem]{Corollary}
\newtheorem{lemma}[theorem]{Lemma}
\newtheorem{remark}[theorem]{Remark}
\newtheorem{proposition}[theorem]{Proposition}

\numberwithin{equation}{section}

\def\C{{\mathbb C}}

\def\N{{\mathbb N}}

\def\R{{\mathbb R}}

\def\E{{\mathbb E}}

\def\eps{\varepsilon}

\newcommand{\abs}[1]{\left\lvert#1\right\rvert}

\begin{document}

\title{Expected number of real zeros for random orthogonal polynomials}

\author{Doron S. Lubinsky, Igor E. Pritsker and Xiaoju Xie}

%\subjclass[2010]{Primary: 30C15; Secondary: 30B20, 60B10.}

\date{}

\maketitle

\begin{abstract}
We study the expected number of real zeros for random linear combinations of orthogonal polynomials. It is well known that Kac polynomials, spanned by monomials with i.i.d. Gaussian coefficients, have only  $(2/\pi + o(1))\log{n}$ expected real zeros in terms of the degree $n$. If the basis is given by the orthonormal polynomials associated with a compactly supported Borel measure on the real line, or associated with a Freud weight defined on the whole real line, then random linear combinations have $n/\sqrt{3} + o(n)$ expected real zeros. We prove that the same asymptotic relation holds for all random orthogonal polynomials on the real line associated with a large class of weights, and give local results on the expected number of real zeros. We also show that the counting measures of properly scaled zeros of these random polynomials converge weakly to  either the Ullman distribution or the arcsine distribution.
\end{abstract}

%Polynomials, random coefficients, expected number of real zeros, random orthogonal polynomials.

\section{Introduction}

The expected number of real zeros $\E[N_n(\R)]$ for random polynomials of the form $P_n(x)=\sum_{k=0}^{n} c_k x^k,$ where $\{c_k\}_{k=0}^n$ are independent and identically distributed random variables, was studied since the 1930's. In particular, Bloch and P\'olya \cite{BP} gave an upper bound $\E[N_n(\R)] = O(\sqrt{n})$ for polynomials with coefficients selected from the set $\{-1,0,1\}$ with equal probabilities. Littlewood and Offord \cite{LO1}-\cite{LO2} considered several classes of random coefficients, including standard Gaussian and uniformly distributed in $\{-1,1\}$ or in the interval $(-1,1).$  Their results indicated that the expected number of real zeros is actually of logarithmic order in terms of $n$. Shortly thereafter, Kac \cite{Ka1} established the important asymptotic result
\[
\E[N_n(\R)] = (2/\pi + o(1))\log{n}\quad\mbox{as }n\to\infty,
\]
for polynomials with independent real Gaussian coefficients. In fact, Kac \cite{Ka1}-\cite{Ka2} found the exact formula for $\E[N_n(\R)]$ in the case of standard real Gaussian coefficients:
\[
\E[N_n(\R)]=\frac{4}{\pi}\displaystyle \int_0^1 \frac{\sqrt{A(x)C(x)-B^2(x)}}{A(x)} \, dx,
\]
where
\[
A(x)=\sum_{j=0}^n x^{2j},\quad
B(x)=\sum_{j=1}^n j x^{2j-1}\quad\mbox{and}\quad
C(x)=\sum_{j=1}^n j^2 x^{2j-2}.
\]
More precise forms of Kac's asymptotic for $\E[N_n(\R)]$ were obtained by many authors, including Kac \cite{Ka2}, Wang \cite{Wa}, Edelman and Kostlan \cite{EK}.  Wilkins \cite{Wi1} gave an asymptotic series expansion for $\E[N_n(\R)]$ in the case of i.i.d. Gaussian coefficients.

The asymptotic result for the number of real zeros was generalized by Erd\H{o}s and Offord \cite{EO} to coefficients with Bernoulli distribution (uniform on $\{-1,1\}$), and by Kac \cite{Ka3} to uniformly distributed coefficients on $[-1, 1]$. Later, Stevens \cite{St} showed that Kac's asymptotic holds for polynomials with coefficients from a certain general class of distributions. Finally, Ibragimov and Maslova \cite{IM1,IM2} extended the result to all mean-zero distributions in the domain of attraction of the normal law. Many additional references and further directions of work on the expected number of real zeros may be found in the books of Bharucha-Reid and Sambandham \cite{BRS}, and of Farahmand \cite{Fa}.

We state a result on the number of real zeros for the random linear combinations of rather general functions. It originated in the papers of Kac \cite{Ka1}-\cite{Ka3}, who used the monomial basis, and was extended to trigonometric polynomials and other bases, see Das \cite{Da1}-\cite{Da2} and Farahmand \cite{Fa}. We are particularly interested in the bases of orthonormal polynomials, which is the case considered by Das \cite{Da1}. Further generalizations of Kac's integral formula for the expected number of real zeros were obtained by several authors, see e.g. Cram\'{e}r and Leadbetter \cite[p. 285]{CL}. For any set $E\subset\C$, we use the notation $N_n(E)$ for the number of zeros of random functions \eqref{1.1} (or random orthogonal polynomials of degree at most $n$) located in $E$. The expected number of zeros in $E$ is denoted by $\E[N_n(E)],$ with $\E[N_n(a,b)]$ being the expected number of zeros in $(a,b)\subset\R.$

\begin{proposition} \label{prop1.1}
Let $[a,b]\subset\R$, and consider real valued functions $g_{j}(x)\in C^1([a,b]), \ j=0,\ldots,n,$ with $g_{0}(x)$ being a nonzero constant. Define the random function
\begin{equation} \label{1.1}
G_n(x)=\sum_{j=0}^{n} c_{j}g_{j}(x),
\end{equation}
where the coefficients $c_j$ are i.i.d. random variables with Gaussian distribution $\mathcal{N}(0, \sigma^2), \sigma > 0$. If there is $M\in\N$ such that $G_n'(x)$ has at most $M$ zeros in $(a,b)$ for all choices of coefficients, then the expected number of real zeros of $G_n(x)$ in the interval $(a,b)$ is given by
\begin{equation} \label{1.2}
\E[N_n(a,b)]=\frac{1}{\pi} \int_a^b \frac{\sqrt{A(x)C(x)-B^2(x)}}{A(x)} \, dx,
\end{equation}
where
\begin{align} \label{1.3}
A(x)=\sum_{j=0}^n g_j^{2}(x), \quad B(x)=\sum_{j=1}^n g_j(x)g_j'(x) \quad \mbox{and}\quad C(x)=\sum_{j=1}^n [g_j'(x)]^2.
\end{align}
\end{proposition}

Clearly, the original formula of Kac follows from this proposition for $g_j(x)=x^j,\ j=0,1,\ldots,n.$ For a sketch of the proof of Proposition \ref{prop1.1}, see \cite{LPX}. We note that multiple zeros are counted only once by the standard convention in all of the above results on real zeros. However, the probability of having a multiple zero for a polynomial with Gaussian coefficients is equal to 0, so that we have the same result on the expected number of zeros regardless whether they are counted with or without multiplicities.

\section{Random orthogonal polynomials}
Let $W=e^{-Q}$, where $Q: \R \longrightarrow [0, \infty)$ is continuous, and assume that all moments
\[
\int_{\R} x^jW^2(x) \, dx, \ j=0, 1, 2, \ldots,
\]
are finite. For $n\geq 0$, let
\begin{equation*}
p_{n}\left( x \right) =p_{n}\left( W^2, x \right) =\gamma _{n}x^{n} + \ldots
\end{equation*}
denote the $n$th orthonormal polynomial with $\gamma_n>0,$ so that
\begin{equation*}
\int p_{n}(x)p_{m}(x)W^2(x)\,dx = \left\{
                                    \begin{array}{ll}
                                      1, & \ m=n, \\
                                      0, & \ m\neq n.
                                    \end{array}
                                  \right.
\end{equation*}
Using the orthonormal polynomials $\{p_j\}_{j=0}^{\infty}$ as the basis, we consider the ensemble of random polynomials of the form
\begin{equation} \label{2.1}
P_n(x)=\sum_{j=0}^{n}c_{j}p_{j}(x),\quad n\in\N,
\end{equation}
where the coefficients $c_0,c_1,\ldots,c_n$ are i.i.d. random variables. Such a family is often called random orthogonal polynomials. If the coefficients have Gaussian distribution, one can apply Proposition \ref{prop1.1} to study the expected number of real zeros of random orthogonal polynomials. In particular, Das \cite{Da1} considered random Legendre polynomials, and found that $\E[N_n(-1,1)]$ is asymptotically equal to $n/\sqrt{3}$. Wilkins \cite{Wi2} improved the error term in this asymptotic relation by showing that $\E[N_n(-1,1)] = n/\sqrt{3} + o(n^\eps)$ for any $\eps>0.$ For random Jacobi polynomials, Das and Bhatt \cite{DB} concluded that $\E[N_n(-1,1)]$ is asymptotically equal to $n/\sqrt{3}$ too. They also stated estimates for the expected number of real zeros of random Hermite and Laguerre polynomials, but those arguments contain significant gaps. The authors recently showed \cite{LPX} that if the basis is given by the orthonormal polynomials associated to a finite Borel measure with compact support on the real line, then random linear combinations have $n/\sqrt{3} + o(n)$ expected real zeros under mild conditions on the weight. The second and the third authors \cite{PX} extended this asymptotic to random orthogonal polynomials associated with the Freud weights $W(x)=e^{-c|x|^\lambda},\ c>0,\ \lambda>1$. The results of this paper provide detailed information on the expected number of real zeros for random polynomials associated with a large class of weights defined on the whole real line. In particular, they cover the case of random Freud polynomials considered in \cite{PX}.

For the orthonormal polynomials $\{p_j(x)\}_{j=0}^{\infty}$, define the reproducing kernel by
\[
K_n(x,y)=\sum_{j=0}^{n-1}p_j(x)p_j(y),
\]
and the differentiated kernels by
\[
K_n^{(k,l)}(x,y)=\sum_{j=0}^{n-1}p_j^{(k)}(x)p_j^{(l)}(y),\quad k,l\in\N\cup\{0\}.
\]
We intend to apply Proposition \ref{prop1.1} with $g_j=p_j$, so that
\begin{align} \label{2.2}
A(x) = K_{n+1}(x,x), \quad B(x) = K_{n+1}^{(0,1)}(x,x) \quad \mbox{and}\quad C(x) = K_{n+1}^{(1,1)}(x,x).
\end{align}
Universality limits for the reproducing kernels of orthogonal polynomials (see Levin and Lubinsky \cite{LL2}-\cite{LL3}), and asymptotic results on zeros of random polynomials (cf. Pritsker \cite{Pr}) allow us to give asymptotics for the expected number of real zeros for a class of random orthogonal polynomials associated with weights from the class $\mathcal{F}(C^2)$ \cite{LL1}.
\begin{definition}
Let $W=e^{-Q}$, where $Q: \R \rightarrow [0, \infty)$ satisfies the following conditions:\\
(a) $Q'$ is continuous in $\R$ and $Q(0)=0$.\\
(b) $Q'$ is non-decreasing in $\R$, and $Q''$ exists in $\R\setminus\{0\}$.\\
(c) \[\displaystyle \lim_{\abs{t}\rightarrow \infty} Q(t)=\infty.\]
(d) The function
\[
T(t)=\frac{tQ'(t)}{Q(t)},\ t \neq 0,
\]
is quasi-increasing in $(0, \infty)$, in the sense that for some $C_1>0$,
\[
0<x<y \Rightarrow T(x) \leq C_1 T(y).
\]
We assume an analogous restriction for $y<x<0$. In addition, we assume that for some $\Lambda >1$,
\[
T(t)\geq \Lambda \text{ in } \R\setminus\{0\}.
\]
(e) There exists $C_2>0$ such that
\[
\displaystyle \frac{Q''(x)}{\abs{Q'(x)}}\leq C_2\frac{\abs{Q'(x)}}{Q(x)}, \ x\in \R\setminus\{0\}.
\]
Then we write $W \in \mathcal{F} (C^2)$.
\end{definition}

Our main result on the global asymptotic for the expected number of real zeros of random orthogonal polynomials is below.
\begin{theorem} \label{thm2.2}
Let $W=e^{-Q}\in \mathcal{F} (C^2)$, where $Q$ is even. If the function $T$ in the definition of $\mathcal{F} (C^2)$ satisfies
\begin{equation}\label{2.3}
\lim_{x\rightarrow \infty}T(x)=\alpha\in(1, \infty],
\end{equation}
then the expected number of real zeros of random orthogonal polynomials \eqref{2.1} with independent real Gaussian coefficients satisfies
\begin{align} \label{2.4}
\lim_{n\to\infty} \frac{1}{n} \E[N_n(\R)]= \frac{1}{\sqrt{3}}.
\end{align}
\end{theorem}

Theorem \ref{thm2.2} is a combination of Theorem \ref{thm2.3} and Corollary \ref{cor2.5} given below. Define the Ullman distribution $\mu_{\alpha}$ for $0<\alpha<\infty$ by
\[
\mu_{\alpha}^{'}(x)=\frac{\alpha}{\pi}\int_{\abs{x}}^1\frac{t^{\alpha-1}}{\sqrt{t^2-x^2}}\,dt,\quad x\in[-1,1],
\]
and for $\alpha=\infty$, the arcsine distribution $\mu_{\infty}$ by
\[
\mu_{\infty}^{'}(x)=\frac{1}{\pi\sqrt{1-x^2}},\quad x\in[-1,1],
\]
see \cite{ST} and \cite{LL1}. We use the contracted version of $P_n$:
\begin{align} \label{2.5}
P_n^*(s):=P_n(a_ns),\quad n\in\N,
\end{align}
where $a_n$ is the Mhaskar-Rakhmanov-Saff number associated with the weight $W$, see \cite{LL1}, \cite{Mh}, \cite{ST} and Section 3 below.

For any set $E\subset\C$, $N_n^*(E)$ denotes the number of zeros of a random polynomial $P_n^*(s)$ located in $E$. The expected number of zeros of $P_n^*(s)$ in $E$ is given by $\E[N_n^*(E)].$ We now state the local result on the asymptotic of $\E[N_n^*\left([a,b]\right)]$ for intervals $[a,b]\subset(-1,1).$

\begin{theorem} \label{thm2.3}
Let $W=e^{-Q} \in \mathcal{F} (C^2)$, where $Q$ is even. Assume that the function $T$ in the definition of $\mathcal{F} (C^2)$ satisfies \eqref{2.3}. If $[a,b] \subset (-1,1)$ is any closed interval, then
\begin{equation} \label{2.6}
\lim_{n\rightarrow\infty} \frac{1}{n}\E\left[N_n^*\left([a,b]\right)\right] = \frac{1}{\sqrt{3}} \mu_{\alpha}([a,b]).
\end{equation}
\end{theorem}

We will establish a generalization of Theorem 2.3 for non-even weights in Section 3. Define the normalized zero counting measure $\tau_n=\frac{1}{n}\sum_{k=1}^n \delta_{z_k}$ for the scaled  polynomial $P_n^*(s)$ of \eqref{2.5}, where $\{z_k\}_{k=1}^n$ are its zeros, and $\delta_z$ denotes the unit point mass at $z$. We determine the weak limit of $\tau_n$ for random polynomials with quite general random coefficients $\{c_j\}_{j=0}^\infty.$

\begin{theorem} \label{thm2.4}
Let the coefficients $\{c_j\}_{j=0}^\infty$ of random orthogonal polynomials \eqref{2.1} be complex i.i.d. random variables such that $\E[|\log|c_0||]<\infty$. If $W=e^{-Q}\in \mathcal{F} (C^2)$, where $Q$ is even, and if the function $T$ in the definition of $\mathcal{F} (C^2)$ satisfies \eqref{2.3}, then the normalized zero counting measures $\tau_n$ for the scaled  polynomials $P_n^*(s)$ converge weakly to $\mu_{\alpha}$ with probability one.
\end{theorem}
Related results on the asymptotic zeros distribution of random orthogonal polynomials with varying weights were proved by Bloom \cite{Bl} and Bloom and Levenberg \cite{BL}. Theorem \ref{thm2.4} permits us to find asymptotics for the expected number of zeros in various sets. In particular, we need the following corollary for the proof of Theorem \ref{thm2.2}.

\begin{corollary} \label{cor2.5}
Suppose that the assumptions of Theorem \ref{thm2.4} hold. If $E\subset\C$ is any compact set satisfying $\mu_\alpha(\partial E)=0,$ then
\begin{equation} \label{2.7}
\lim_{n\rightarrow\infty} \frac{1}{n}\E\left[N_n^*(E)\right] = \mu_\alpha(E).
\end{equation}
\end{corollary}

It is of interest to relax conditions on random coefficients $c_j$, e.g., by considering probability distributions from the domain of attraction of the normal law as in \cite{IM1,IM2}.

\section{Proofs}

Our proofs require detailed knowledge of potential theory with external fields, see \cite{LL1} and \cite{ST}.

Let $W$ be a continuous nonnegative weight function on $\R$ such that $W$ is not identically zero and $\lim_{|x|\to\infty} |x|\,W(x) = 0.$ Set $Q(x):= - \log W(x).$ The weighted equilibrium measure $\mu _W$ of $\R$  is the unique probability measure with compact support $S_W=\text{ supp } \mu_W \subset \R$ that minimizes the energy functional
\begin{equation*}
I[\nu] = \iint \frac{1}{|z-t|}\, d\nu(t) d\nu(z)+2\int Q\, d\nu
\end{equation*}
amongst all probability measures $\nu$ with support on $\R$. It satisfies
\[
\int \log \frac{1}{|z-t|}\, d\mu_W(t) +Q(z)=C,\quad z \in S_W,
\]
and
\[
\int \log \frac{1}{|z-t|}\, d\mu_W(t) +Q(z)\geq C,\quad z \in \R,
\]
where $C$ is a constant.

For a weight function $W(x)=e^{-Q(x)}$, where $Q$ is often assumed convex on $\R$, the Mhaskar-Rakhmanov-Saff numbers
\[
a_{-n}<0<a_n
\]
are defined for $n \geq 1$ by the relations
\[
n=\frac{1}{\pi}\int_{a_{-n}}^{a_n}\frac{xQ'(x)}{\sqrt{(x-a_{-n})(a_n-x)}} \, dx
\]
and
\[
0=\frac{1}{\pi}\int_{a_{-n}}^{a_n}\frac{Q'(x)}{\sqrt{(x-a_{-n})(a_n-x)}} \, dx.
\]
We also let
\[
\delta_n=\frac{1}{2}(a_n+\abs{a_{-n}})\text{ and }\beta_n=\frac{1}{2}(a_n+a_{-n}).
\]
For even $Q$, $a_{-n}=-a_n$, and we may define $a_n$ by
\begin{equation}\label{mrsnumber}
\frac{2}{\pi}\int_0^1\frac{a_ntQ'(a_nt)}{\sqrt{1-t^2}}\, dt=n.
\end{equation}
Existence and uniqueness of these numbers are established in the monographs \cite{LL1}, \cite{Mh}, \cite{ST}, but go back to earlier work of Mhaskar, Saff, and Rakhmanov. One illustration of their role is the Mhaskar-Saff identity:
\[
||PW||_{L^\infty (\R)}=||PW||_{L^\infty ([a_{-n}, a_n])},
\]
which is valid for all polynomials $P$ of degree at most $n$. %It is known that the Mhaskar-Rakhmanov-Saff number associated with the Freud weight $W(x)=e^{-c\abs{x}^\lambda}$ is given by
%\[
%a_n=\gamma_\lambda^{\frac{1}{\lambda}}c^{-\frac{1}{\lambda}}n^{\frac{1}{\lambda}}.
%\]
%See page 308 of \cite{ST} for further details.
We define the Mhaskar-Rakhmanov-Saff interval $\Delta_n$ as $\Delta_n :=[a_{-n}, a_n]$. The linear transformation
\[
L_n(x)=\frac{x-\beta_n}{\delta_n},\ x\in \R,
\]
maps $\Delta_n$ onto $[-1, 1]$. Its inverse is
\[
L_n^{[-1]}(s)=\beta_n+\delta_ns,\ s\in \R.
\]
For $\varepsilon\in(0,1)$, we let
\[
J_n(\varepsilon)=L_n^{[-1]}[-1+\varepsilon, 1-\varepsilon]=[a_{-n}+\varepsilon\delta_n, a_n-\varepsilon\delta_n].
\]
The equilibrium density is defined as
\[
\sigma_n(x)=\displaystyle \frac{\sqrt{(x-a_{-n})(a_n-x)}}{\pi^2}\int_{a_{-n}}^{a_n} \frac{Q'(s)-Q'(x)}{s-x}\frac{ds}{\sqrt{(s-a_{-n})(a_n-s)}}, \, x\in \Delta_n.
\]
It satisfies the following equilibrium equations \cite[p. 41]{LL1}:
\[
\displaystyle \int_{a_{-n}}^{a_n} \log \frac{1}{\abs{x-s}}\sigma_n(s)\, ds + Q(x)=C, \, x\in \Delta_n,
\]
and
\[
\displaystyle \int_{a_{-n}}^{a_n} \log \frac{1}{\abs{x-s}}\sigma_n(s)\, ds + Q(x)\geq C, \, x\in \R.
\]
Note that the measure $\sigma_n(x)\, dx$ has total mass $n$ on $\Delta_n$:
\[
\displaystyle \int_{a_{-n}}^{a_n} \sigma_n(x)\, dx=n.
\]
We also define the normalized version of $\sigma_n$ as follows:
\begin{align*}
\sigma_n^*(s) := \frac{\delta_n}{n}\sigma_n(L_n^{[-1]}(s)),\quad s\in [-1,1].
\end{align*}
Note that $\sigma_n^*(s)\, ds$ is a unit measure supported on $[-1,1]$:
\[
\int_{-1}^1\sigma_n^*(s) \, ds=1.
\]
For details on $\sigma_n$ and $\sigma_n^*$, one should consult the book \cite{LL1}.

In particular, the Ullman distribution $\mu'_{\alpha}$ is the normalized equilibrium density for the standard Freud weight $w(x)=e^{-\gamma_\alpha \abs{x}^\alpha}$ on $\R$, see Theorem 5.1 of \cite[p. 240]{ST}, where
\[
\gamma_\alpha=\frac{\Gamma(\frac{\alpha}{2})\Gamma(\frac{1}{2})}{2\Gamma(\frac{\alpha}{2}+\frac{1}{2})},
\]
An alternative formula for the Ullman distribution follows from that for $\sigma_n$ above, namely,
\begin{equation}\label{alternative}
\mu_{\alpha}^{'}(x)=\frac{2\sqrt{1-x^2}}{\pi^2B_{\alpha}}\int_0^1\frac{t^\alpha-x^\alpha}{t^2-x^2}\frac{dt}{\sqrt{1-t^2}}, \quad x\in[-1,1],
\end{equation}
where
\[
B_\alpha=\frac{2}{\pi}\int_0^1\frac{t^\alpha}{\sqrt{1-t^2}}\, dt.
\]

For $n\geq 1$, we also define the square root factor
\begin{equation}\label{rho}
\rho_n(x)=\sqrt{(x-a_{-n})(a_n-x)},\quad x\in \Delta_n.
\end{equation}
In the sequel $C, C_1, C_2, \cdots$ denote constants independent of $n, x$, and polynomials of degree $\leq n$. The same symbol does not necessarily denote the same constant in different occurrences. Given sequences $\{c_n\}, \{d_n\}$, we write
\[
c_n \sim d_n
\]
if there exist positive constants $C_1$ and $C_2$ such that for $n \geq 1$,
\[
C_1\leq c_n/d_n\leq C_2.
\]
Similar notation is used for functions and sequences of functions.

We start with a general result, our only one that allows non-even weights. In this more general setting, $P_n^*$ is given by
\[
P_n^*(s) = P_n\left(L_n^{[-1]}(s)\right),
\]
rather then by \eqref{2.5}.

\begin{theorem} \label{thm3.1}
If $W=e^{-Q}\in \mathcal{F} (C^2)$ and $[a,b] \subset (-1,1)$ is any given closed interval, then as $n \to \infty$,
\begin{equation*}
\frac{1}{n}\E\left[N_n^*\left([a,b]\right)\right] = \frac{1+o(1)}{\sqrt{3}} \int_a^b \sigma_{n+1}^*(y)\, dy.
\end{equation*}
\end{theorem}

\begin{proof}
The strategy is to apply Theorem 1.6 of \cite{LL2}. It states that for all $r,s\geq 0,$ and any $\varepsilon \in (0, 1)$, we have uniformly for $x \in J_n(\varepsilon)$ as $n \to \infty$,
\begin{equation*}
\frac{W^2(x)K_n^{(r,s)}(x,x)}{(\sigma_n(x))^{r+s+1}}=\sum_{j=0}^r \left( \begin{array}{c} r \\ j \end{array}  \right) \sum_{k=0}^s \left( \begin{array}{c} s \\ k \end{array}  \right)\tau_{j,k}\pi^{j+k}\left(\frac{Q'(x)}{\sigma_n(x)}\right)^{r+s-j-k}+o(1),
\end{equation*}
where
\begin{equation*}
\tau _{j,k}=\left\{
\begin{tabular}{ll}
$0,$ & $j+k$ odd, \\
$(-1)^{(j-k)/2}\frac{1}{j+k+1},$ & $j+k$ even.
\end{tabular}
\right.
\end{equation*}
In particular, uniformly in $x \in J_{n+1}(\varepsilon)$,
\[
\frac{W^2(x)K_{n+1}^{\left(0,0\right)}\left(x,x\right)}{\sigma_{n+1}(x)}=1+o(1),\] \[\frac{W^2(x)K_{n+1}^{\left(0,1\right)}\left(x,x\right)}{(\sigma_{n+1}(x))^{2}}=\frac{Q'(x)}{\sigma_{n+1}(x)}+o(1),\]
and
\[\frac{W^2(x)K_{n+1}^{\left(1,1\right)}\left(x,x\right)}{(\sigma_{n+1}(x))^{3}}=\left(\frac{Q'(x)}{\sigma_{n+1}(x)}\right)^2+\frac{\pi^2}{3}+o(1).
\]
Next, from Proposition 1.1, for any closed interval $[l, q] \subset J_{n+1}(\varepsilon)$ (where $l, q$ may depend on $n$),
\begin{equation*}
\frac{1}{n} \E\left[N_{n}\left([l, q]\right)\right] = \frac{1}{n\pi }\int_l^q \sqrt{\frac{K_{n+1}^{\left( 1,1\right) }\left( x,x\right) }{K_{n+1}^{\left( 0,0\right) }\left( x,x\right)}-\left( \frac{K_{n+1}^{\left(0,1\right) }\left(x,x\right)}{K_{n+1}^{\left( 0,0\right) }\left( x,x\right) }\right) ^{2}}\, dx.
\end{equation*}
Substituting the above asymptotics, and cancelling, yields
\begin{align*}
&\frac{1}{n}\E\left[ N_{n}([l, q])\right]  \\
&=\frac{1}{n\pi } \int_l^q \sigma_{n+1}(x) \sqrt{\frac{\pi^2}{3}+\left(\frac{Q'(x)}{\sigma_{n+1}(x)}\right)^2 o(1) + \frac{Q'(x)}{\sigma_{n+1}(x)}o(1)+o(1)}\, dx \quad\mbox{as } n\to\infty.
\end{align*}
We note that \cite[p. 87, Lemma 5.1(a),(d)]{LL2} uniformly for $x\in J_{n+1}(\varepsilon)$,
\[
\sigma_{n+1}(x)\geq C_1\frac{n+1}{\delta_{n+1}}
\]
and
\[
\abs{Q'(x)}\leq C_2\frac{n+1}{\rho_{n+1}(x)},
\]
so that
\[
\abs{\frac{Q'(x)}{\sigma_{n+1}(x)}}\leq C_3 \frac{\delta_{n+1}}{\rho_{n+1}(x)}\leq \frac{C_3}{\sqrt{\eps(2-\eps)}},\quad x\in J_{n+1}(\varepsilon).
\]
Thus, uniformly for all intervals $[l, q]\subset J_{n+1}(\varepsilon)$, as $n\to\infty$,
\begin{align*}
\frac{1}{n}\E\left[ N_{n}([l, q])\right]
=\frac{1}{n\pi } \int_l^q \sigma_{n+1}(x) \sqrt{\frac{\pi^2}{3}+ o(1)}\, dx=\frac{1+o(1)}{n\sqrt{3}}\int_l^q \sigma_{n+1}(x) \, dx.
\end{align*}
Note that the number $N_{n}(E)$ of real zeros of $P_n(x)$ in $E$ equals the number $N_{n}^*(E^*)$ of real zeros of $P_n^*(s)$ in $E^*:=L_n(E)=\{L_n(x): x\in E\}$, since $L_n$ is a bijection. We recall that $a_n$ is increasing to $+\infty$ and $a_{-n}$ is decreasing to $-\infty$ as $n\to\infty.$ It is also known that
\begin{align*}
\lim_{n\to\infty}\frac{a_{n+1}}{a_n}=1, \quad \lim_{n\to\infty}\frac{a_{-(n+1)}}{a_{-n}}=1 \quad \mbox{and} \quad \lim_{n\to\infty}\frac{\delta_{n+1}}{\delta_n}=1,
\end{align*}
see Lemma 3.11(a) of \cite[p. 81]{LL1}. Hence we have
\begin{align} \label{bijection}
L_{n+1}\left(L_n^{[-1]}(s)\right) = L_{n+1}(\beta_n + \delta_n s) = \frac{\delta_n}{\delta_{n+1}} s + \frac{\beta_n-\beta_{n+1}}{\delta_{n+1}} \to s \quad\mbox{as } n\to\infty,
\end{align}
uniformly for $s$ in compact subsets of $\R.$ If $[a, b]\subset (-1, 1)$, then for large $n\in\N$,
\[
L_n^{[-1]}([a, b])=[a_{-n}+\delta_n(1+a), a_n-\delta_n(1-b)]\subset J_{n+1}(\varepsilon),
\]
provided $0<\varepsilon<\min \{1+a, 1-b \}$. It follows that
\begin{align*}
\frac{1}{n}\E\left[ N_{n}^*([a,b])\right]&=\frac{1}{n}\E\left[ N_{n}\left(L_n^{[-1]}([a,b])\right)\right]\\
&=\frac{1+o(1)}{(n+1)\sqrt{3}} \int_{L_n^{[-1]}(a)}^{L_n^{[-1]}(b)} \sigma_{n+1}(x)\,dx\\
&=\frac{1+o(1)}{\sqrt{3}} \int_{L_{n+1}\left(L_n^{[-1]}(a)\right)}^{L_{n+1}\left(L_n^{[-1]}(b)\right)} \sigma_{n+1}^*(s)\,ds \quad (\mbox{where } s=L_{n+1}(x))\\ &=\frac{1+o(1)}{\sqrt{3}} \int_a^b \sigma_{n+1}^*(s)\,ds \quad\mbox{as } n\to\infty,
\end{align*}
where we used \eqref{bijection} on the last step, and also that
\[
\sigma_{n+1}^*(s) \le \frac{C}{\sqrt{1-s^2}}, \quad s\in(-1,1),
\]
by Theorem 1.11(V) of \cite[p. 18]{LL1}.
\end{proof}

\begin{lemma} \label{lem3.2}
Let $W=e^{-Q}\in \mathcal{F} (C^2)$, where $Q$ is even. Let $\alpha\in(1, \infty]$. If the function $T$ in the definition of $\mathcal{F} (C^2)$ satisfies
\begin{equation*}
\lim_{x\rightarrow \infty}T(x)=\alpha,
\end{equation*}
then
\[
\lim_{n\rightarrow \infty} \sigma_n^*(x)=\mu_{\alpha}^{'}(x),\quad x\in (-1,1)\setminus\{0\}.
\]

\end{lemma}

\begin{remark} \label{rem3.3}
An equivalent form of
\[
\lim_{x\rightarrow \infty}T(x)=\alpha\in(1, \infty)
\]
is that uniformly for $t$ in compact subsets of $(0,1],$
\begin{equation}\label{equiv}
\lim_{x\rightarrow \infty}\frac{Q'(xt)}{Q'(x)}=t^{\alpha-1}.
\end{equation}
Indeed, if this last condition holds, then as $x\rightarrow \infty$,
\begin{align*}
T(x)^{-1}&=\frac{Q(x)}{xQ'(x)}=\frac{1}{xQ'(x)}\int_0^{x}Q'(u)\, du\\
&=\int_0^1\frac{Q'(xt)}{Q'(x)}\, dt\rightarrow \int_0^1 t^{\alpha-1}\, dt=\frac{1}{\alpha}.
\end{align*}
Here we also used $0\leq Q'(xt)/Q'(x)\leq 1$ and dominated convergence. In the other direction, as $x\rightarrow \infty$,
\begin{align*}
\frac{Q'(xt)}{Q'(x)}&=\frac{T(xt)}{T(x)}\frac{Q(xt)}{tQ(x)} = \frac{T(xt)}{tT(x)}\exp\left(-\int_{xt}^{x} \frac{Q'(u)}{Q(u)}\,du\right)\\
&=\frac{T(xt)}{tT(x)}\exp\left(-\int_{xt}^{x} \frac{T(u)}{u}\,du\right)\\
&=\frac{T(xt)}{tT(x)}\exp\left(-\int_{xt}^{x} \frac{\alpha+o(1)}{u}\,du\right)\\
&=\frac{1+o(1)}{t}\exp\left(-(\alpha+o(1))\log \frac{1}{t}\right)=t^{\alpha-1}(1+o(1)).
\end{align*}
Given any $\eps\in(0,1),$ this holds uniformly for $t\in [\eps,1].$
\end{remark}

\begin{proof}[Proof of Lemma 3.2]
We prove the case $1 <\alpha <\infty$ first:\\
From (\ref{mrsnumber}), as $n\rightarrow \infty$,
\begin{align}
\frac{n}{a_nQ'(a_n)}&=\frac{2}{\pi}\int_0^1\frac{tQ'(a_nt)}{Q'(a_n)\sqrt{1-t^2}}\, dt\nonumber\\
&\rightarrow \frac{2}{\pi}\int_0^1 \frac{t^{\alpha}}{\sqrt{1-t^2}}\, dt=B_{\alpha}. \label{3.2.0}
\end{align}
Indeed, the integrand converges pointwise, and because $Q$ is convex, so $Q'(a_nt)/Q'(a_n)\leq 1$, and we can apply Lebesgue's Dominated Convergence Theorem. In particular, for $n \geq 1$, and some $C_1>1$ independent of $n$,
\begin{equation}\label{3.2.7}
C_1^{-1}n\leq a_nQ'(a_n)\leq C_1n.
\end{equation}
Next, we know that for $x\in (0,1)$,
\[
\sigma_n^*(x)=\frac{2\sqrt{1-x^2}}{\pi^2}\int_0^1\frac{a_ntQ'(a_nt)-a_nxQ'(a_nx)}{n(t^2-x^2)}\frac{dt}{\sqrt{1-t^2}}.
\]
For $t\in (0,1)\setminus \{x\}$, we obtain from \eqref{equiv} and \eqref{3.2.0} that
\begin{align*}
&\lim_{n\rightarrow \infty}\frac{a_ntQ'(a_nt)-a_nxQ'(a_nx)}{n(t^2-x^2)}\\
&=B_{\alpha}^{-1}\lim_{n\rightarrow \infty}\frac{a_ntQ'(a_nt)-a_nxQ'(a_nx)}{a_nQ'(a_n)(t^2-x^2)}\\
&=B_{\alpha}^{-1}\frac{t^{\alpha}-x^{\alpha}}{t^2-x^2}.
\end{align*}
We need a bound on the integrand so as to apply dominated convergence. First, $T(u)$ is bounded above. Next, for some $\xi$ between $t$ and $x$,
\begin{align*}
&\abs{\frac{a_ntQ'(a_nt)-a_nxQ'(a_nx)}{n(t^2-x^2)}}\\
&=\abs{\frac{\frac{d}{du}(a_nuQ'(a_nu))\vert_{u=\xi}}{n(t+x)}}\\
&\leq \frac{a_nQ'(a_n\xi)+a_n^2\xi Q''(a_n\xi)}{n(t+x)}.
\end{align*}
Here (\ref{3.2.7}) gives (since $Q'$ is increasing)
\[
\frac{a_nQ'(a_n\xi)}{n(t+x)}
\leq \frac{a_nQ'(a_n)}{n(t+x)}\leq\frac{C_1}{x}.
\]
By definition of $\mathcal{F}(C^2)$ and boundedness of $T$, we have
\[
0\leq \frac{Q''(y)}{Q'(y)}\leq \frac{C_2 T(y)}{y}\leq \frac{C_3}{y},\quad y>0,
\]
so that
\[
\frac{a_n^2\xi Q''(a_n\xi)}{n(t+x)}\leq C_3\frac{a_n^2\xi Q'(a_n\xi)}{a_n\xi n(t+x)}\leq C_3\frac{a_nQ'(a_n\xi)}{n(t+x)}\leq \frac{C_4}{x}.
\]
Thus, for all $t \in (0, 1)$,
\[
\abs{\frac{a_ntQ'(a_nt)-a_nxQ'(a_nx)}{n(t^2-x^2)}}\leq \frac{C_5}{x},
\]
and we can apply dominated convergence to deduce that
\[
\lim_{n\rightarrow \infty}\sigma_n^*(x)=\frac{2\sqrt{1-x^2}}{\pi^2 B_{\alpha}}\int_0^1\frac{t^\alpha-x^\alpha}{t^2-x^2}\frac{dt}{\sqrt{1-t^2}}=\mu_{\alpha}^{'}(x).
\]
Next, we deal with the case $\alpha=\infty$:\\
Let $0<r<s<1$. We consider $x\in (0, r]$ and split
\begin{align}
\sigma_n^*(x)&=\frac{2\sqrt{1-x^2}}{\pi^2}\left(\int_0^s+\int_s^1\right)\frac{a_ntQ'(a_nt)-a_nxQ'(a_nx)}{n(t^2-x^2)}\frac{dt}{\sqrt{1-t^2}}\nonumber \\
&=:I_1+I_2. \label{3.2.11}
\end{align}
We shall show that the main contribution to $\sigma_n^*$ comes from $I_2$. Since the integrand in the integral defining $\sigma_n^*$ is nonnegative, we have for $x\in (0, r]$ that
\begin{align}
I_2&=\frac{2\sqrt{1-x^2}}{\pi^2}\int_s^1\frac{a_ntQ'(a_nt)-a_nxQ'(a_nx)}{n(t^2-x^2)}\frac{dt}{\sqrt{1-t^2}}\nonumber\\
&\leq \frac{2\sqrt{1-x^2}}{\pi^2}\int_s^1\frac{a_ntQ'(a_nt)}{n(t^2-x^2)}\frac{dt}{\sqrt{1-t^2}}\nonumber\\
&\leq  \frac{2\sqrt{1-x^2}}{\pi^2(s^2-x^2)n}\int_s^1 a_ntQ'(a_nt)\frac{dt}{\sqrt{1-t^2}}\nonumber\\
&\leq \frac{\sqrt{1-x^2}}{\pi(s^2-x^2)n} \frac{2}{\pi}\int_0^1 a_ntQ'(a_nt)\frac{dt}{\sqrt{1-t^2}}\nonumber\\
&=\frac{\sqrt{1-x^2}}{\pi(s^2-x^2)}.\label{3.2.12}
\end{align}
Next, note that by the lower bound in (3.5) of \cite[p. 64]{LL1}, for $t\in[0, r]$,
\begin{align*}
0&\leq \frac{a_ntQ'(a_nt)}{a_nsQ'(a_ns)}\leq \frac{a_nrQ'(a_nr)}{a_nsQ'(a_ns)}\leq \frac{T(a_nr)}{T(a_ns)}\left(\frac{r}{s} \right)^{\max \{\Lambda, C_6 T(a_nr)\} }\\
&\leq C_7 \left(\frac{r}{s} \right)^{C_8 T(a_nr)},
\end{align*}
since $T$ is quasi-increasing. Our hypothesis
\[
\lim_{x\rightarrow \infty}T(x)=\infty
\]
gives
\begin{equation}\label{3.2.13}
\lim_{n\rightarrow \infty} \max_{t\in [0, r]} \frac{a_ntQ'(a_nt)}{a_nsQ'(a_ns)}=0.
\end{equation}
It also then follows easily from (\ref{mrsnumber}) that for each fixed $\tau \in (0, 1)$,
\begin{equation} \label{3.2.14}
\lim_{n\rightarrow \infty} \frac{a_n\tau Q'(a_n\tau)}{n}=0.
\end{equation}
Now uniformly for $x\in [0, r]$,
\begin{align} \label{3.2.15}
I_2&\geq \frac{2\sqrt{1-x^2}}{\pi^2(1-x^2)}\int_s^1 \frac{a_ntQ'(a_nt)-a_nxQ'(a_nx)}{n}\frac{dt}{\sqrt{1-t^2}}\nonumber\\
&\geq \frac{1}{\pi \sqrt{1-x^2}}\frac{2}{\pi n}\int_s^1 a_ntQ'(a_nt)(1+o(1))\frac{dt}{\sqrt{1-t^2}}\nonumber\\
&=\frac{1+o(1)}{\pi \sqrt{1-x^2}}\frac{2}{\pi n}\int_0^1 a_ntQ'(a_nt)\frac{dt}{\sqrt{1-t^2}} \nonumber\\
&=\frac{1+o(1)}{\pi \sqrt{1-x^2}} \quad\mbox{as } n\to\infty,
\end{align}
by (\ref{mrsnumber}) and using (\ref{3.2.13}). Now we deal with $I_1$ - it clearly suffices to show only an upper bound. Let $s<\rho<1$. By definition of the class $\mathcal{F}(C^2)$ and (\ref{3.2.14}), we have that
\begin{align*}
I_1&=\frac{2\sqrt{1-x^2}}{\pi^2}\int_0^s \frac{a_ntQ'(a_nt)-a_nxQ'(a_nx)}{n(t^2-x^2)}\frac{dt}{\sqrt{1-t^2}}\\
&\leq \frac{2\sqrt{1-x^2}}{\pi^2 nx} \max_{u\in [0, s]} \abs{\frac{d}{du}(a_nuQ'(a_nu))}\int_0^s\frac{dt}{\sqrt{1-t^2}}\\
&\leq \frac{C_9}{nx}[a_nQ'(a_ns)+\max_{u\in [0, s]} a_n^2uQ''(a_nu)]\\
&\leq o(1)+\frac{C_9}{nx}\max_{u\in [0, s]} a_nQ'(a_nu)T(a_nu) \quad\mbox{as } n\to\infty.
\end{align*}
Using the fact that $T$ is quasi-increasing and the lower bound in (3.5) of \cite[p. 64]{LL1}, we continue this as
\begin{align*}
I_1 &\leq o(1)+\frac{C_9}{nx} a_nQ'(a_ns)T(a_ns)\\
&\leq o(1)+\frac{C_9}{nx} a_nQ'(a_n\rho)\frac{T(a_ns)}{T(a_n\rho)}\left(\frac{s}{\rho}\right)^{\max \{\Lambda, C_6 T(a_ns)\}-1}T(a_ns)\\
&\leq o(1)+\frac{C_9}{nx} a_nQ'(a_n\rho)\sup_{y\in [0, \infty)}\left(\frac{s}{\rho}\right)^{\max \{\Lambda, C_6 y\}-1} y=o(1)\ \mbox{as } n\to\infty,
\end{align*}
by (\ref{3.2.14}) and as $s/\rho <1$. Together with the fact that $I_1 \geq 0$, and using (\ref{3.2.11}), (\ref{3.2.12}), (\ref{3.2.15}), we have shown that for $x\in (0, r]$,
\[
\frac{1}{\pi \sqrt{1-x^2}}\leq \liminf_{n\rightarrow \infty} \sigma_n^*(x)\leq \limsup_{n\rightarrow \infty} \sigma_n^*(x)\leq \frac{\sqrt{1-x^2}}{\pi (s^2-x^2)}.
\]
As $s$ is independent of $r$, we can let $s\rightarrow 1-$ to deduce that for $x\in (0, r]$,
\[
\lim_{n\rightarrow \infty} \sigma_n^*(x)=\frac{1}{\pi \sqrt{1-x^2}}=\mu_{\infty}^{'}(x).
\]
\end{proof}

\begin{proof}[Proof of Theorem \ref{thm2.3}]
We know from Theorem \ref{thm3.1} that
\[
\frac{1}{n}\E\left[N_n^*\left([a,b]\right)\right] = \frac{1+o(1)}{\sqrt{3}} \int_a^b \sigma_{n+1}^*(y)\, dy.
\]
Lemma \ref{lem3.2} gives for $1<\alpha \leq \infty$ that
\[
\lim_{n\rightarrow \infty} \sigma_{n+1}^*(y)=\mu_{\alpha}^{'}(y),\quad y\in (-1,1)\setminus\{0\}.
\]
Next, by Theorem 1.11(V) of \cite[p. 18]{LL1},
\[
\sigma_{n+1}^*(s) \le \frac{C}{\sqrt{1-s^2}}, \quad s\in(-1,1).
\]
Lebesgue's Dominated Convergence Theorem now implies that
\begin{align*}
\lim_{n\rightarrow \infty}\frac{1}{n}\E\left[N_n^*\left([a,b]\right)\right] &= \frac{1}{\sqrt{3}} \int_a^b \lim_{n\rightarrow \infty} \sigma_{n+1}^*(y)\, dy=\frac{1}{\sqrt{3}}\mu_\alpha([a, b]).
\end{align*}
\end{proof}

\begin{lemma}\label{lem3.4}
 If $W=e^{-Q}\in \mathcal{F} (C^2)$ then
\[\lim_{n\to\infty}a_n^{1/n} = 1.\]
\end{lemma}

\begin{proof}
Lemma 3.5(c) of \cite[p. 72]{LL1} implies that there is a constant $C>0$ such that
\[
1\leq \frac{a_n}{a_1}\leq Cn^{1/\Lambda} \text{ for all }n\geq 1,
\]
which immediately gives the needed result.
\end{proof}

\begin{lemma}\label{lem3.5}
Let $W=e^{-Q}\in \mathcal{F} (C^2)$, where $Q$ is  even. If the coefficients $\{c_j\}_{j=0}^\infty$ of random orthogonal polynomials \eqref{2.1} are complex i.i.d. random variables such that $\E[|\log|c_0||]<\infty$, then
\[
\lim_{n\to\infty} \left\|P_n W \right\|_{L^{\infty}(\R)}^{1/n} = 1 \text{ with probability one}.
\]
\end{lemma}

\begin{proof}
Using orthogonality, we obtain for polynomials defined in \eqref{2.1} that
\[
\int_{-\infty}^\infty |P_n(x)|^2 W^2(x)\,dx = \sum_{j=0}^n |c_j|^2.
\]
Hence
\[
\max_{0\le j\le n} |c_j| \le \left(\int_{-\infty}^\infty |P_n(x)|^2 W^2(x)\,dx \right)^{1/2} \le (n+1) \max_{0\le j\le n} |c_j|.
\]
Lemma 4.2 of \cite{Pr} (see (4.6) there) implies that
\[
\lim_{n\to\infty} \left(\int_{-\infty}^\infty |P_n(x)|^2 W^2(x)\,dx \right)^{1/(2n)} = \lim_{n\to\infty} \left( \max_{0\le j\le n} |c_j| \right)^{1/n} = 1
\]
with probability one. That is,
\begin{equation}\label{3.5.1}
\lim_{n\to\infty} \left\|P_n W \right\|_{L^{2}(\R)}^{1/n} = 1 \text{ with probability one}.
\end{equation}
We use the Nikolskii inequalities of Theorem 10.3 of \cite[p. 295]{LL1} stated as
\[
\left\|P_n W \right\|_{L^{\infty}(\R)} \leq C_1 \left(\frac{n}{a_n}\right)^{1/2}(T(a_n))^{1/4} \left\|P_n W \right\|_{L^2(\R)}
\]
and
\[
\left\|P_n W \right\|_{L^2(\R)} \leq C_2  a_n^{1/2} \left\|P_n W \right\|_{L^{\infty}(\R)}.
\]
Since $T(a_n)=O(n^2)$ by Lemma 3.7 of \cite[p. 76]{LL1}, we obtain that
\[
\frac{1}{C_2}\frac{1}{\sqrt{a_n}}\left\|P_n W \right\|_{L^2(\R)}\leq \left\|P_n W \right\|_{L^{\infty}(\R)}\leq C_3 n\left\|P_n W \right\|_{L^2(\R)},
\]
and the result follows by applying Lemma \ref{lem3.4} and (\ref{3.5.1}).
\end{proof}

\begin{lemma}\label{lem3.6}
Let $W=e^{-Q}\in \mathcal{F} (C^2)$, where $Q$ is even. If the function $T$ in the definition of $\mathcal{F} (C^2)$ satisfies
\begin{equation*}
\lim_{x\rightarrow \infty}T(x)=\infty,
\end{equation*}
then
\[
\lim_{n\to\infty} \gamma_n^{1/n} a_n = 2,
\]
where $\gamma_n$ is the leading coefficient of the orthonormal polynomial $p_n(x)$ associated with the weight $W^2.$
\end{lemma}

\begin{proof}
Theorem 1.22 of \cite[p. 25]{LL1} gives
\[
\gamma_n=\frac{1}{\sqrt{2\pi}}\left(\frac{a_n}{2}\right)^{-n-\frac{1}{2}}e^{\frac{1}{\pi}\int_{-a_n}^{a_n}\frac{Q(s)}{\sqrt{a_n^2-s^2}} \, ds}(1+o(1)) \text{ as }n \to \infty,
\]
so that
\begin{equation}\label{root}
\gamma_n^{1/n} a_n = 2 a_n^{-\frac{1}{2n}} e^{\frac{1}{n\pi}\int_{-a_n}^{a_n}\frac{Q(s)}{\sqrt{a_n^2-s^2}} \, ds}(1+o(1)) \text{ as }n \to \infty.
\end{equation}
Since $Q$ is increasing on $(0,\infty)$, we have that
\begin{align} \label{3.6.1}
0 \le \lim_{n\to\infty} \frac{1}{n\pi}\int_{-a_n}^{a_n}\frac{Q(s)}{\sqrt{a_n^2-s^2}} \, ds \le \frac{Q(a_n)}{n\pi} \int_{-a_n}^{a_n} \frac{ds}{\sqrt{a_n^2-s^2}} = \frac{Q(a_n)}{n} \le \frac{C}{\sqrt{T(a_n)}} \to 0
\end{align}
as $n\to\infty$, by  Lemma 3.4 of \cite[p. 69]{LL1}. Thus
\[
\lim_{n\to\infty} \frac{1}{n\pi}\int_{-a_n}^{a_n}\frac{Q(s)}{\sqrt{a_n^2-s^2}} \, ds = 0,
\]
and \eqref{root} together with Lemma \ref{lem3.4} imply the result.
\end{proof}

​
\begin{lemma}\label{lem3.7}
Let $W=e^{-Q}\in \mathcal{F} (C^2)$, where $Q$ is even. If the function $T$ in the definition of $\mathcal{F} (C^2)$ satisfies
\begin{equation*}
\lim_{x\rightarrow \infty}T(x)=\alpha \in (1, \infty),
\end{equation*}
then
\[
\lim_{n\to\infty}  \gamma_n^{1/n} a_n = 2e^{1/\alpha},
\]
where $\gamma_n$ is the leading coefficient of the orthonormal polynomial $p_n(x)$ associated with the weight $W^2.$
\end{lemma}

\begin{proof}
Considering Lemma \ref{lem3.4} and (\ref{root}), we only need to show
\[
\lim_{n\to\infty} \frac{1}{n\pi}\int_{-a_n}^{a_n}\frac{Q(s)}{\sqrt{a_n^2-s^2}} \, ds=
\lim_{n\to\infty} \frac{1}{n}\int_{-1}^{1}\frac{Q(a_nt)}{\pi\sqrt{1-t^2}} \, dt=1/\alpha.
\]
In terms of the function $T$, we can recast this as
\begin{equation*}
\lim_{n\to\infty} \frac{1}{n}\int_{-1}^{1} \frac{1}{T(a_nt)} \frac{a_ntQ'(a_nt)}{\pi\sqrt{1-t^2}} \, dt=1/\alpha.
\end{equation*}
Using our assumption that $\lim_{t\rightarrow \infty}T(t)=\alpha \in (1, \infty)$, we have uniformly for $\abs{t}\geq a_n^{-1/2}$, that $T(a_nt)=\alpha(1+o(1))$, so as the integrand is non-negative,
\begin{equation}\label{3.7.1}
\frac{1}{n}\int_{a_n^{-1/2}\leq \abs{t}\leq 1}\frac{1}{T(a_nt)} \frac{a_ntQ'(a_nt)}{\pi\sqrt{1-t^2}} \, dt=\frac{1+o(1)}{\alpha}\frac{1}{n}\int_{a_n^{-1/2}\leq \abs{t}\leq 1}\frac{a_ntQ'(a_nt)}{\pi\sqrt{1-t^2}} \, dt.
\end{equation}
The integral over the remaining range is small: for $j=0, 1$, using \eqref{3.2.0} and $\lim_{n\to \infty} a_n=\infty$,
\begin{align*}
0 &\leq \frac{1}{n}\int_{\abs{t}\leq a_n^{-1/2}}\frac{1}{T(a_nt)^j} \frac{a_ntQ'(a_nt)}{\pi\sqrt{1-t^2}} \, dt\\
&\leq \frac{1}{n}\frac{a_n^{1/2}Q'(a_n^{1/2})}{\Lambda^j \pi \sqrt{1-a_n^{-1}}}2a_n^{-1/2}\leq C\frac{Q'(a_n^{1/2})}{n}\leq C\frac{Q'(a_n)}{n}=o(1).
\end{align*}
Thus \eqref{3.7.1} and \eqref{mrsnumber} yield
\[
 \frac{1}{n}\int_{-1}^{1} \frac{1}{T(a_nt)} \frac{a_ntQ'(a_nt)}{\pi\sqrt{1-t^2}} \, dt=\frac{1+o(1)}{\alpha}
 \frac{1}{n} \int_{-1}^{1} \frac{a_ntQ'(a_nt)}{\pi\sqrt{1-t^2}} \, dt=\frac{1+o(1)}{\alpha}.
\]
\end{proof}

\begin{proof}[Proof of Theorem \ref{thm2.4}]
We first deal with the case
\[
\lim_{x\rightarrow \infty}T(x)=\infty,
\]
and show that the normalized zero counting measures $\tau_n$ for the scaled polynomials $P_n^*(s)$ converge weakly to the arcsine distribution $\mu_\infty$ with probability one. Theorem 2.1 of \cite[p. 310]{BSS} states that if $\{M_n\}_{n=1}^\infty$  is any sequence of monic polynomials of degree deg$(M_n)=n$ satisfying
\begin{equation}\label{2.4.1}
\limsup_{n\to \infty} \left\|M_n \right\|_{L^\infty([-1, 1])}^{1/n}\leq \frac{1}{2},
\end{equation}
then the normalized zero counting measures $\tau_n$ for the polynomials $M_n$ converge weakly to $\mu_\infty$. Note that $1/2$ in the above equation is the logarithmic capacity of $[-1,1]$, see Corollary 5.2.4 of \cite[p. 134]{Ra}.
We show that the monic polynomials
\[
M_n(x) := P_n^*(x)/(c_n\gamma_n a_n^n),\quad n\in\N,
\]
satisfy (\ref{2.4.1}) with probability one, so that the result of Theorem \ref{thm2.4} follows for $\alpha=\infty.$ We know from Lemma \ref{lem3.5} that
\[
\limsup_{n\to\infty} \left\|P_n W \right\|_{L^\infty(\R)}^{1/n} \leq 1 \text{ with probability one.}
\]
Using the contracted weight
\[
w_n(s):=\sqrt[n]{W(a_ns)}=e^{-\frac{Q(a_ns)}{n}},\quad s \in \R,
\]
and the properties of $a_n$ \cite[p. 4]{LL1}, we obtain that
\[
\left\|P_n^* w_n^n \right\|_{L^\infty([-1, 1])}=\left\|P_nW \right\|_{L^\infty([-a_n, a_n])}=\left\|P_nW \right\|_{L^\infty(\R)}.
\]
It follows that
\[
\limsup_{n\to\infty} \left\|P_n^* w_n^n \right\|_{L^\infty([-1, 1])}^{1/n}\leq 1 \text{ with probability one.}
\]
Since $\lim_{n\to \infty}Q(a_n)/n=0$ (recall (\ref{3.6.1})), we have that
\begin{align*}
\limsup_{n\to\infty} \left\|P_n^*\right\|_{L^\infty([-1, 1])}^{1/n} &\leq \limsup_{n\to\infty} \left\|P_n^*w_n^n\right\|_{L^\infty([-1, 1])}^{1/n} e^{Q(a_n)/n} \leq 1
\end{align*}
with probability one. We use below that $\lim_{n\to\infty} \gamma_n^{1/n} a_n=2$ by Lemma \ref{lem3.6}, and that  $\lim_{n\to\infty} |c_n|^{1/n} = 1$ with probability one by Lemma 4.2 of \cite{Pr}. This implies that
\begin{align*}
\limsup_{n\to\infty} \left\|M_n\right\|_{L^\infty([-1,1])}^{1/n} &= \limsup_{n\to\infty} \left\|\frac{P_n^*}{c_n\gamma_na_n^n}\right\|_{L^\infty([-1,1])}^{1/n} \\ &= \limsup_{n\to\infty} \left\|P_n^*\right\|_{L^\infty([-1,1])}^{1/n} \frac{1}{|c_n|^{1/n}} \frac{1}{\gamma_n^{1/n}a_n} \le \frac{1}{2}
\end{align*}
with probability one.

Next, we prove the case
\[
\lim_{x\rightarrow \infty}T(x)=\alpha\in (1, \infty).
\]
Recall that the standard Freud weight with index $\alpha$ is given by
\[
w(s)=e^{-\gamma_\alpha\abs{s}^\alpha}, \quad s\in \R,
\]
where
\[
\gamma_\alpha = \frac{\Gamma(\frac{\alpha}{2})\Gamma(\frac{1}{2})}{2\Gamma(\frac{\alpha}{2}+\frac{1}{2})} = \int_0^1 \frac{t^{\alpha-1}}{\sqrt{1-t^2}}\, dt
\]
see \cite[p. 239]{ST}. Since $\gamma_{\alpha+1}=B_\alpha \pi/2,$ we apply $\Gamma(1/2)=\sqrt{\pi}$ and $\Gamma(t+1)=t\Gamma(t)$ to obtain that
\begin{align*}
\gamma_\alpha B_\alpha& = \gamma_\alpha\frac{2\gamma_{\alpha+1}}{\pi} = \frac{2}{\pi}\frac{\Gamma(\frac{\alpha}{2})\Gamma(\frac{1}{2})}{2\Gamma(\frac{\alpha}{2}+\frac{1}{2})}
\frac{\Gamma(\frac{\alpha+1}{2})\Gamma(\frac{1}{2})}{2\Gamma(\frac{\alpha+1}{2}+\frac{1}{2})}
=\frac{1}{\alpha}.
\end{align*}
Note that by \cite[p. 240]{ST}, $F_w=\log 2+1/\alpha$ is the modified Robin constant and $\mu_w=\mu_\alpha$ is the equilibrium measure corresponding to $w$. Following \cite{ST}, we call a sequence of monic polynomials $\{M_n\}_{n=1}^\infty$, with $ \deg(M_n)=n,$ asymptotically extremal with respect to the weight $w$ if it satisfies
\begin{equation}\label{2.4.2}
\lim_{n\to\infty} \|w^n M_n\|_{L^\infty(\R)}^{1/n} = e^{-F_w} = e^{-1/\alpha}/2.
\end{equation}
Theorem 4.2 of \cite[p. 170]{ST} states that  asymptotically extremal monic polynomials have their zeros distributed according to the measure $\mu_w$. Namely, the normalized zero counting measures of $M_n$ converge weakly to $\mu_w=\mu_\alpha$. On the other hand, by Corollary 2.6 of \cite[p. 157]{ST} and Theorem 5.1 of \cite[p. 240]{ST},
\[
\|w^n M_n\|_{L^\infty(\R)}=\|w^n M_n\|_{L^\infty([-1,1])}.
\]
Together with Theorem 3.6 of \cite[p. 46]{ST}, (\ref{2.4.2}) is equivalent to
\[
\limsup_{n\to\infty} \|w^n M_n\|_{L^\infty([-1, 1])}^{1/n} \leq e^{-F_w} = e^{-1/\alpha}/2.
\]
We show that the monic polynomials
\[
M_n(x) := P_n^*(x)/(c_n\gamma_n a_n^n),\quad n\in\N,
\]
are asymptotically extremal in this sense with probability one, so that the result of Theorem \ref{thm2.4} follows. Note that
\[
\lim_{n\to\infty} \left\|P_n W \right\|_{L^\infty(\R)}^{1/n} = 1 \text{ with probability one}
\]
by Lemma \ref{lem3.5},
and that
\[
\left\|P_n^* w_n^n \right\|_{L^\infty([-1, 1])} = \left\|P_nW \right\|_{L^\infty([-a_n, a_n])} = \left\|P_nW \right\|_{L^\infty(\R)}
\]
by \cite[p. 4]{LL1}.
Hence
\[
\limsup_{n\to\infty} \left\|P_n^* w_n^n \right\|_{L^\infty([-1, 1])}^{1/n}\leq 1 \text{ with probability one.}
\]
By Lemma \ref{lem3.7}, and since $\lim_{n\to\infty} |c_n|^{1/n} = 1$ with probability one by Lemma 4.2 of \cite{Pr}, it follows that
\begin{align*}
\limsup_{n\to\infty} \left\|M_n w^n \right\|_{L^\infty([-1, 1])}^{1/n} &= \limsup_{n\to\infty} \left\|P_n^*w^n \right\|_{L^\infty([-1, 1])}^{1/n}\frac{1}{c_n^{1/n} \gamma_n^{1/n}a_n} \\ &= \frac{1}{2e^{1/\alpha}}\limsup_{n\to\infty} \left\|P_n^* w^n \right\|_{L^\infty([-1, 1])}^{1/n} \\ &=e^{-F_w}\limsup_{n\to\infty} \left\|P_n^* w^n \right\|_{L^\infty([-1, 1])}^{1/n}.
\end{align*}
On the other hand,
\begin{align*}
\limsup_{n\to\infty} \left\|P_n^* w^n \right\|_{L^\infty([-1, 1])}^{1/n} &\le \limsup_{n\to\infty} \left\|P_n^* w_n^n \right\|_{L^\infty([-1, 1])}^{1/n} \left\|w/w_n\right\|_{L^\infty([-1, 1])} \\ &\le \limsup_{n\to\infty} \left\|w/w_n\right\|_{L^\infty([-1,1])}.
\end{align*}
Since $w_n$ and $w$ are both even, it remains to show that
\[
\limsup_{n\to\infty} \left\|w/w_n\right\|_{L^\infty([0,1])} \le 1.
\]
Let $\eps\in(0,1)$. For $x\in[\eps,1],$ \eqref{3.2.0} and then \eqref{equiv} give that
\begin{align*}
\frac{Q(a_ns)}{n} &= \frac{1+o(1)}{B_\alpha}\int_0^s \frac{a_nQ'(a_nx)}{a_nQ'(a_n)}\,dx = \frac{1+o(1)}{B_\alpha} \int_0^s x^{\alpha-1}(1+o(1))\,dx \\ &= \frac{s^\alpha}{\alpha B_\alpha} (1+o(1)) = \gamma_\alpha s^\alpha (1+o(1)) \quad \mbox{as } n\to\infty.
\end{align*}
This holds uniformly for $s\in[\eps,1]$ as \eqref{equiv} does. Hence
\[
\left\|w/w_n\right\|_{L^\infty([\eps,1])} =  \sup_{s\in[\eps,1]} \exp\left(\frac{Q(a_ns)}{n} - \gamma_\alpha s^\alpha\right) \to 1\quad \mbox{as } n\to\infty.
\]
Since $Q$ is increasing, we also have that
\[
\left\|w/w_n\right\|_{L^\infty([0,\eps])} \le  \exp\left(\frac{Q(a_n\eps)}{n}\right) \to \exp(\gamma_\alpha \eps^\alpha).
\]
We finish the proof by letting $\eps\to 0.$
\end{proof}

\begin{proof}[Proof of Corollary \ref{cor2.5}]
Consider the normalized zero counting measure $\tau_n=\frac{1}{n}\sum_{k=1}^n \delta_{z_k}$ for the scaled  polynomial $P_n^*(s)$ of \eqref{2.5}, where $\{z_k\}_{k=1}^n$ are the zeros of that polynomial, and $\delta_z$ denotes the unit point mass at $z$. Theorem \ref{thm2.4} implies that the measures $\tau_n$ converge weakly to $\mu_\alpha$ with probability one. Since $\mu_\alpha(\partial E)=0,$ we obtain that $\tau_n\vert_E$ converges weakly to $\mu_\alpha\vert_E$ with probability one by Theorem $0.5^\prime$ of \cite{La} and Theorem 2.1 of \cite{Bi}. In particular, we have that the random variables $\tau_n(E)$ converge to $\mu_\alpha(E)$ with probability one. Hence this convergence holds in $L^p$ sense by the Dominated Convergence Theorem, as $\tau_n(E)$ are uniformly bounded by 1, see Chapter 5 of \cite{Gut}. It follows that
\[
\lim_{n\to\infty} \E[|\tau_n(E) - \mu_\alpha(E)|] = 0
\]
for any compact set $E$ such that $\mu_\alpha(\partial E)=0,$ and
\[
\left|\E[\tau_n(E) - \mu_\alpha(E)]\right| \le \E[|\tau_n(E) - \mu_\alpha(E)|] \to 0 \quad\text{as } n\to\infty.
\]
But $\E[\tau_n(E)]=\E[N_n^*(E)]/n$ and $\E[\mu_\alpha(E)]=\mu_\alpha(E),$ which immediately gives \eqref{2.7}.
\end{proof}

\begin{proof}[Proof of Theorem \ref{thm2.2}]
Theorem \ref{thm2.3} gives that
\[
\lim_{n\rightarrow\infty} \frac{1}{n}\E\left[N_n^*\left([a,b]\right)\right] = \frac{1}{\sqrt{3}} \mu_\alpha([a,b])
\]
for any interval $[a,b]\subset (-1,1)$. Note that both $\E\left[N_n^*\left(H\right)\right]$ and $\mu_\alpha(H)$ are additive functions of the set $H$. Moreover, they both vanish when $H$ is a single point by \eqref{2.7} and the absolute continuity of $\mu_\alpha$ with respect to Lebesgue measure on $[-1,1]$.
Hence \eqref{2.7} gives that
\[
\lim_{n\rightarrow\infty} \frac{1}{n}\E\left[N_n^*\left(\R\setminus (-1,1)\right)\right] = \mu_\alpha(\R\setminus(-1,1)) = 0.
\]
It now follows that
\[
\lim_{n\to\infty} \frac{1}{n} \E[N_n^*(\R)] = \frac{1}{\sqrt{3}} \mu_\alpha((-1,1)) = \frac{1}{\sqrt{3}}.
\]
To complete the proof, observe that $N_n^*(\R)=N_n(\R)$, so that $\E[N_n^*(\R)]=\E[N_n(\R)]$, since $L_n(x)=x/a_n$ is a bijection for each fixed $n$.
Therefore (\ref{2.4}) is proved.
\end{proof}

\section*{Acknowledgements}

Research of the first author was partially supported by NSF grant DMS136208. Research of the second author was partially supported by the National Security Agency (grant H98230-15-1-0229) and by the American Institute of Mathematics. Work of the third author is done towards completion of her Ph.D. degree at Oklahoma State University under the direction of the second author.

\bigskip
\textsc{Doron S. Lubinsky}\\ School of Mathematics, Georgia Institute of Technology, Atlanta, GA 30332, USA \\ e-mail\textup{: \texttt{lubinsky@math.gatech.edu}}

\medskip
\textsc{Igor E. Pritsker}\\ Department of Mathematics, Oklahoma State University, Stillwater, OK 74078, USA \\ e-mail\textup{: \texttt{igor@math.okstate.edu}}

\medskip
\textsc{Xiaoju Xie}\\ Department of Mathematics, Oklahoma State University, Stillwater, OK 74078, USA \\ e-mail\textup{: \texttt{sophia.xie@okstate.edu}}


\begin{thebibliography}{00}

\bibitem{BRS} A. T. Bharucha-Reid and M. Sambandham, Random Polynomials, Academic Press, Orlando, 1986.

\bibitem{Bi} P. Billingsley, Convergence of Probability Measures, John Wiley \& Sons, Inc., New York, 1999.

\bibitem{BSS} H.-P. Blatt, E. B. Saff, and M. Simkani, Jentzsch-Szeg\H{o} type theorems for the zeros of best approximants, J. London Math. Soc. 38 (1988), 307--316.

\bibitem{BP}  A. Bloch and G. P\'olya, On the roots of certain algebraic equations, Proc. London Math. Soc. 33 (1932), 102--114.

\bibitem{Bl} T. Bloom, Random polynomials and (pluri)potential theory, Ann. Polon. Math. 91 (2007), 131--141.

\bibitem{BL} T. Bloom and N. Levenberg, Random polynomials and pluripotential-theoretic extremal functions, Potential Anal. 42 (2015), 311--334.

\bibitem{CL} H. Cram\'{e}r and M. R. Leadbetter, Stationary and Related Stochastic Processes, Wiley, New York, 1966.

\bibitem{Da1} M. Das, Real zeros of a random sum of orthogonal polynomials, Proc. Amer. Math. Soc. 27 (1971), 147--153.

\bibitem{Da2} M. Das, The average number of real zeros of a random trigonometric polynomial, Proc. Camb. Phil. Soc. 64 (1968), 721--729.

\bibitem{DB} M. Das and S. S. Bhatt, Real roots of random harmonic equations, Indian J. Pure Appl. Math. 13 (1982), 411--420.

\bibitem{EK} A. Edelman and E. Kostlan, How many zeros of a random polynomial are
real?, Bull. Amer. Math. Soc. 32 (1995), 1--37.

\bibitem{EO}  P. Erd\H{o}s and A. C. Offord, On the number of real roots of a random algebraic equation, Proc. London Math. Soc. 6 (1956), 139--160.

\bibitem{Fa} K. Farahmand, Topics in Random Polynomials, Pitman Res. Notes Math. 393 (1998).

\bibitem{Fa1} K. Farahmand, Level crossings of a random orthogonal polynomial, Analysis 16 (1996), 245--253.

\bibitem{Fa2} K. Farahmand, On random orthogonal polynomials, J. Appl. Math. Stochastic Anal. 14 (2001), 265--274.

\bibitem{Fr} G. Freud, Orthogonal Polynomials, Akademiai Kiado/Pergamon Press, Budapest, 1971.

\bibitem{Gut} A. Gut, Probability: A Graduate Course, Springer, New York, 2005.

\bibitem{IM1} I. A. Ibragimov and N. B. Maslova, The average number of zeros of random polynomials, Vestnik Leningrad University 23 (1968), 171--172.

\bibitem{IM2} I. A. Ibragimov and N. B. Maslova, The mean number of real zeros of random polynomials. I. Coefficients with zero mean, Theory Probab. Appl. 16 (1971), 228--248.

\bibitem{Ka1} M. Kac, On the average number of real roots of a random algebraic equation, Bull. Amer. Math. Soc. 49 (1943), 314--320.

\bibitem{Ka2} M. Kac, On the average number of real roots of a random algebraic equation. II, Proc. London Math. Soc. 50 (1948), 390--408.

\bibitem{Ka3} M. Kac, Nature of probability reasoning, Probability and related topics in physical sciences, Proceedings of the Summer Seminar, Boulder, Colo., 1957, Vol. I Interscience Publishers, London-New York, 1959.

\bibitem{La} N. S. Landkof, Foundations of Modern Potential Theory, Springer-Verlag, New York - Heidelberg, 1972.

\bibitem{LL1} E. Levin and D. S. Lubinsky, Orthogonal Polynomials for Exponential Weights, Springer, New York, 2001.

\bibitem{LL2} E. Levin and D. S. Lubinsky, Applications of universality limits to zeros and reproducing kernels of orthogonal polynomials, J. Approx. Theory 150 (2008), 69--95.

\bibitem{LL3} E. Levin and D. S. Lubinsky, Universality limits for exponential weights, Constr. Approx. 29 (2009), 247--275.

\bibitem{LO1} J. E. Littlewood and A. C. Offord, On the number of real roots of a random algebraic equation, J. London Math. Soc. 13 (1938), 288--295.

\bibitem{LO2} J. E. Littlewood and A. C. Offord, On the number of real roots of a random algebraic equation. II, Proc. Camb. Philos. Soc. 35 (1939), 133--148.

\bibitem{LPX} D. S. Lubinsky, I. E. Pritsker, and X. Xie, Expected number of real zeros for random linear combinations of orthogonal polynomials, Proc. Amer. Math. Soc., to appear. arXiv: 1503.06376

\bibitem{Mh} H. N. Mhaskar, Introduction to the Theory of Weighted Polynomial Approximation, World Scientific, Singapore, 1996.

\bibitem{MS} H. N. Mhaskar and E. B. Saff, Extremal problems for polynomials with exponential weights, Trans. Amer. Math. Soc. 285 (1984), 203--234.

\bibitem{Pr} I. E. Pritsker, Zero distribution of random polynomials, J. Anal. Math., to appear. arXiv:1409.1631

\bibitem{PX} I. E. Pritsker and X. Xie, Expected number of real zeros for random orthogonal polynomials, J. Math. Anal. Appl., to appear. arXiv:1505.04762

\bibitem{Ra} T. Ransford, Potential Theory in the Complex Plane, Cambridge Univ. Press, Cambridge, 1995.

\bibitem{ST} E. B. Saff and V. Totik, Logarithmic Potentials with External Fields, Springer, New York, 1997.

\bibitem{St} D. C. Stevens, The average number of real zeros of a random polynomial, Comm. Pure Appl. Math. 22 (1969), 457--477.

\bibitem{Wa} Y. J. Wang, Bounds on the average number of real roots of a random algebraic equation, Chinese Ann. Math. Ser. A. 4 (1983), 601--605.

\bibitem{Wi1} J. E. Wilkins, Jr., An asymptotic expansion for the expected number of real zeros of a random polynomial, Proc. Amer. Math. Soc. 103 (1988), 1249--1258.

\bibitem{Wi2} J. E. Wilkins, Jr., The expected value of the number of real zeros of a random sum of Legendre polynomials, Proc. Amer. Math. Soc. 125 (1997), 1531--1536.
\end{thebibliography}
\end{document}